\documentclass[12pt,reqno]{amsart}
\usepackage{bm}
\usepackage{graphicx}
\usepackage{adjustbox}
\usepackage{amsthm,amssymb,amsmath}
\usepackage[T1]{fontenc}
\usepackage[utf8]{inputenc}
\usepackage[colorlinks=true, linkcolor=blue, urlcolor=blue, citecolor=blue, anchorcolor=blue]{hyperref}
\usepackage{xcolor}
\usepackage{mathrsfs}
\usepackage{enumitem}
\usepackage{setspace}
\usepackage{yfonts}
\usepackage{float}
\usepackage{mathtools}
\usepackage[all,cmtip]{xy}
\usepackage{todonotes}
\usepackage{tikz-cd}
\usepackage{relsize}

\newtheorem{manualtheoreminner}{Theorem}
\newenvironment{manualtheorem}[1]{%
\renewcommand\themanualtheoreminner{#1}%
  \manualtheoreminner
}{\endmanualtheoreminner}

\usepackage[backend=bibtex]{biblatex}
\tikzcdset{row sep/normal=50pt, column sep/normal=50pt}

\newtheorem{lemma}{Lemma}[section]

\newtheorem{corollary}[lemma]{Corollary}

\theoremstyle{definition}

\newtheorem{remark}[lemma]{Remark}

\numberwithin{equation}{section}

\begin{document}
\title[]{Singularity of cubic hypersurfaces and hyperplane sections of projectivized tangent bundle of projective space}
\author[A. Bansal]{Ashima Bansal}
\address{Department of Mathematics, Shiv Nadar University, NH91, Tehsil Dadri, Greater Noida, Uttar Pradesh 201314, India}
\email{ashima.bansal@snu.edu.in}
\author[S. Sarkar]{Supravat Sarkar}
\address{\scshape Fine Hall, Princeton, NJ 700108}
\email{ss6663@princeton.edu}
\author[S. Vats]{Shivam Vats}
\address{Indian Institute of Science Education and Research Tirupati, Srinivasapuram, Yerpedu Mandal, Tirupati Dist, Andhra Pradesh, India – 517619}
\email{shivamvatsaaa@gmail.com}

\begin{abstract}
We show that the normal points of a cubic hypersurface in projective space have canonical singularities unless the hypersurface is an iterated cone over an elliptic curve. As an application, we give a simple linear algebraic description of all the hyperplane sections of projectivized tangent bundle of projective space, hence describing hyperplane sections of a rational homogeneous manifold of Picard rank $2$. This also simplifies and extends recent results of Mazouni-Nagaraj in higher dimensions. We also compute the Chow ring of these hyperplane sections.
\end{abstract}
\maketitle
\section{Introduction}
We work throughout, over the field $k=\mathbb{C}$ of complex numbers. Given a smooth projective variety $X$ in $\mathbb{P}^N$, it is natural to study the hyperplane sections of $X$, that is, the schemes $H\cap X$ for $H$ a hyperplane in $\mathbb{P}^N$. One is led to investigate the conditions under which the hyperplane section is singular, and if it is singular, how bad the singularity can be. On the one hand, this leads to the notion of dual varieties. On the other hand, since the linear system of hyperplane sections is a flat family, one gets valuable information about singular degenerations of smooth varieties.

If $X$ is the image of $d$-uple embedding of $\mathbb{P}^n$, the hyperplane sections are the degree $d$ hypersurfaces in $\mathbb{P}^n$. A complete classification of these hypersurfaces is known only for some very specific values of $d$ and $n$. For $d=1$, they are all isomorphic to $\mathbb{P}^{n-1}$. For $d=2$, it is easy to see that they are iterated cones over some smooth quadric hypersurface, a union of two distinct hyperplanes, or a nonreduced scheme supported on a hyperplane. For $d=3$ and $n=2$, they are cubic curves in $\mathbb{P}^2$, and it is easy to completely describe them up to isomorphism. For $d=3$ and $n>2$, that is, the study and classification of cubic hypersurfaces has been an interesting and fruitful topic for a long time. The study of smooth cubic hypersurfaces yielded the beautiful result that any smooth cubic surface in $\mathbb{P}^3$ contains exactly $27$ lines, also giving birth to the famous problem of deciding the rationality of cubic fourfolds. A detailed study of smooth cubic hypersurfaces is in \cite{Huy}.

The study of singular cubic hypersurfaces is more complicated and requires studying several cases. Singular cubic surfaces are classified in \cite{Sc}. \cite{BL}, \cite{Loo}, \cite{Cay}, \cite{SS} also studied singular cubic surfaces. \cite{BW1} studied the singularities of cubic surfaces in light of modern singularity theory. In \S 2 of this article, we aim to do the same for higher dimensional cubic hypersurfaces.

There has not been much progress in the study of singularities of higher dimensional cubic hypersurfaces.  \cite{LPS} studied non-normal cubic hypersurfaces. In \S 2 of this article, we prove the following result about singularity in the normal locus of a cubic hypersurface. This result gives some conditions under which a hypersurface has canonical/rational singularities. In that sense, it has a similar flavour as the results in \cite{BMW}, \cite{BW2}.

\begin{manualtheorem}{A}\label{A} 

Let $n$ be a positive integer, $X$ a cubic hypersurface in $\mathbb{P}^n$, and assume that $X$ is not an iterated cone over an elliptic curve. Let $U$ be a nonempty open subvariety of $X$ such that $U$ is normal. Then $U$ has canonical singularities.
\end{manualtheorem}

Here for a projective variety $Z$ in $\mathbb{P}^n$, an iterated cone over $Z$ is the projective variety in $\mathbb{P}^{n+r}$ for some $r\geq 0$ which is the cone over $Z$ with vertex/axis a linear subvariety of dimension $r-1$ (see {\cite[Exercise I.5.12(d)]{Hart}}). 

In \S 3, we study hyperplane sections and their singularities for the projective variety $\mathbb{P}(T_{\mathbb{P}^n})$. It is well-known that the tangent bundle of $\mathbb{P}^n$, $T_{\mathbb{P}^n}$ is ample. Moreover, the line bundle $\mathcal{O}_{\mathbb{P}(T_{\mathbb{P}^n})}(1)$ is very ample, so we get a natural embedding of $\mathbb{P}(T_{\mathbb{P}^n})$ in a projective space of dimension $(n+1)^2-2$. If we look at $\mathbb{P}(T_X)$ for other smooth projective varieties $X$, we do not have such a natural projective embedding, as \cite{Mo} proved $\mathbb{P}^n$ is the only smooth projective variety with ample tangent bundle. It is natural to study the hyperplane sections and their singularity of these special projective varieties $\mathbb{P}(T_{\mathbb{P}^n})$. \cite{MN} studied it for $n=2$. In \S 3 of this paper, we study it in higher dimensions. Using Theorem A, we can give a description of their singularities, which is analogous to the result in \cite{MN}. Our method also gives a simpler and elementary proof to the result in \cite{MN}. A key feature of our result is that the description of the hyperplane sections is given completely in terms of some linear algebraic data related to the hyperplane chosen. This phenomenon is also present and easy to see for the Plucker embedding of Grassmanians and some other rational homogeneous spaces of Picard rank $1$ under the embedding given by the ample generator of the Picard group, or product of some such varieties. For other projective varieties, it is indeed rare to have such a simple linear algebraic description of the hyperplane sections. To the authors' best knowledge, our result is the first result studying hyperplane sections of a full class of rational homogeneous spaces of Picard rank $\geq 2$ which are not products of rational homogeneous spaces of Picard rank $1.$

To state our result, we need the following notations. For a positive integer $m$, denote the space of $m\times m$ matrices by $M_{m\times m}$. For $0\leq s\leq r$, let $$ V_{s,r}:= \left\{ ( [\underline{z}],[\underline{w}]) \in \mathbb{P}^{r-1} \times \mathbb{P}^{r-1} \,\,|\,\, \sum_{s< i \leq r}  z_{i}w_{i}=0 \right\}.$$ \hspace{10pt}Here $[\underline{z}]=[z_1,\ldots,z_r],$ 
 $ [\underline{w}]=[w_1,\ldots,w_r]$ are points of $\mathbb{P}^{r-1}$. Note that $ V_{0,r} \cong \mathbb{P}(T_{\mathbb{P}^{r-1}})$, $ V_{r,r}= \mathbb{P}^{r-1} \times \mathbb{P}^{r-1}.$  $V_{s,r}$ for $s<r$ are exactly the hyperplane sections of $\mathbb{P}^{r-1} \times \mathbb{P}^{r-1}$ under the Segre embedding. We use the convention of projectivization of vector spaces/vector bundles as in {\cite[Chapter 2, section 7]{Hart}}. In particular, for a vector space $W$ with dual $W^*$, the elements of $\mathbb{P}(W^{\ast})$ are the lines in $W$ passing through the origin. We denote the element of $\mathbb{P}(W^{\ast})$ corresponding to the line in $W$ passing through the origin and a nonzero point $w\in W$ by $[w]$.  Also, for a finite type scheme $Y$ over $\mathbb{C}$, let $Y_{\textrm{sm}}, Y_{\textrm{sing}}$ be the smooth locus and singular locus of $Y$, respectively. 
 
 Note that we have the following identification:
 \begin{equation*}
     \mathbb{P}_{\mathbb{P}^n}(T_{\mathbb{P}^n})= \Big\{ ( [\underline{x}], [\underline{y}] \in \mathbb{P}^n \times \mathbb{P}^n\medspace |\medspace \underline{x}^{t} \underline{y}=0 \Big\}.
 \end{equation*} 
 For $\Bar{A}\in M_{(n+1)\times (n+1)} / \textrm{scalar matrices}$, define
 \begin{equation}\label{H}
     H_{[\Bar{A}]}= \Big\{ ([\underline{x}], [\underline{y}]) \in X \,\, |\,\, \underline{x}^{t}  A   \underline{y}=0\Big \}= \Big\{ ([\underline{x}], [\underline{y}]) \in  \hspace{3pt}\mathbb{P}^n \times \mathbb{P}^n | \hspace{3pt} \underline{x}^{t} \underline{y}=\underline{x}^{t} A\underline{y}=0\Big\}.
 \end{equation} 
 This is a prime divisor in $\mathbb{P}(T_{\mathbb{P}^n})$.
 
 Now we can state our result.

\begin{manualtheorem}{B}\label{B} Let $ n \geq 2 $, $W = M_{(n+1)\times (n+1)} / \textrm{scalar matrices} $.
There is a natural isomorphism 
    $$\mathbb{P}(W^{\ast})  \cong    | \mathcal{O}_{\mathbb{P}(T_{\mathbb{P}^n})}(1) |$$
    $$ [\Bar{A}] \longmapsto  H_{[\Bar{A}]} $$ 
such that $H_{[\Bar{A}]}$ is always reduced of degree  $\binom {2n} n$, and we have the following.
\begin{enumerate}
\item  $H_{[\Bar{A}]}$ is not irreducible if and only if $[\Bar{A}] = [\Bar{B}]$ for some rank 1 matrix $B$. In this case, $H_{[\Bar{A}]} $ has two irreducible components, each having degree $ \frac{1}{2} \binom {2n} n$  and isomorphic to $ \mathbb{P}_{\mathbb{P}^{n-1}}(\mathcal{O}_{\mathbb{P}^{n-1}} \oplus T_{\mathbb{P}^{n-1}}(-1))$. Their intersection is isomorphic to 
    \begin{equation}
    \begin{cases}
V_{0,n}, & \quad \textrm{if A is diagonalizable}, \\
V_{1,n}, & \quad\textrm{if A is not diagonalizable }.
\end{cases}
\end{equation}
All non-irreducible $ H_{[\bar{A}]}$'s with diagonalizable $A$'s are isomorphic, all   non-irreducible $ H_{[\bar{A}]}$'s with non-diagonalizable $A$'s are isomorphic.

\item Now suppose $ H_{[\Bar{A}]} $ is irreducible. Then $ H_{[\Bar{A}]} $ is a normal rational Fano variety with singular locus $  (H_{[\Bar{A}]})_{\emph{sing}} \cong \sqcup_{i} V_{s_{i},r_{i}} $, where $s_i$, $r_i$ are defined as follows, for each eigenvalue $ \lambda_{i}$ of $A$, 
\begin{equation*}
   r_{i}= \textrm{ number of Jordan blocks in a Jordan form of A corresponding to eigenvalue } \lambda_{i}
\end{equation*}
and
\begin{equation*}
 s_{i}= \textrm{ number of Jordan blocks in a Jordan form of A of size >1, corresponding to eigenvalue } \lambda_{i}.
 \end{equation*}
Each singularity of $ H_{[\Bar{A}]}$ is a canonical singularity, and is locally isomorphic to the singularity of a normal rational hypersurface in $ \mathbb{A}^{2n-1}$ of degree  $ \leq 3$.
  \item For $ n \geq 3 $ and a general $ [\Bar{A}] \in \mathbb{P}(W^{\ast})$, $H_{[\Bar{A}]} $ is a smooth rational Fano variety of dimension $ 2n-2$, Picard rank 2, whose both elementary contraction have target $ \mathbb{P}^n$ and general fiber $ \mathbb{P}^{n-2}$. 
  \end{enumerate}

\end{manualtheorem}

 The above description allows us to find the dual variety of $\mathbb{P}(T_{\mathbb{P}^n})$ in Corollary \ref{dual}.

Also, extending the results in \cite{MN}, Theorem \ref{B} describes all possible deformations in $\textrm{SL}_{n+1}(\mathbb{C})/B$ of a codimension $1$ subscheme $Y$ which is union of two Schubert divisors, here $\textrm{SL}_{n+1}(\mathbb{C})$ is the group of all $(n+1)\times (n+1)$ matrices over $\mathbb{C}$ of determinant $1$ and $B$ is a parabolic subgroup (see Remark \ref{deform}).

For $ \Bar{A} \in W \backslash 
{0} $ such that $H = H_{[\bar{A}]}$ is smooth,  the Chow ring of $H$ is defined. We compute the Chow ring in Theorem \ref{C}. Let us first introduce some notations. For a smooth projective variety $Y$ of dimension $n$, let $A_{k}(Y)$ be the Chow group of the $k$-cycles in $Y$, and let $A^{k}(Y)=A_{n-k}(Y).$ Under the intersection product, $A^{\ast}(Y)= \oplus_{k=0}^{n} A^{k}(Y)$ becomes a graded ring. Since $H= H_{[\bar{A}]}$ is smooth, by Corollary \ref{dual} $A$ has distinct eigenvalues. So, there exist distinct points $ p_{0},p_{1},\ldots,p_{n}\in \mathbb{P}^n$ corresponding to the eigenvectors of $A$. Let $ H \xhookrightarrow{i} \mathbb{P}(T_{\mathbb{P}^n})$ be the inclusion, $ \mathbb{P}(T_{\mathbb{P}^n}) \xlongrightarrow{q}  \mathbb{P}^n$ the projection, $ \pi=q \circ i : H \longrightarrow \mathbb{P}^n$. Let $ \alpha= \pi^{\ast} c_{1}(\mathcal{O}_{\mathbb{P}^n}(1)) $,  $ \zeta= i^{\ast} c_{1}(\mathcal{O}_{\mathbb{P}(T_{\mathbb{P}^n})}(1)) \in A^{1}(H).$  Note that for $ x\in \mathbb{P}^n$, we have $ \pi^{-1}(x) \cong \mathbb{P}^{n-2}$ if $ x \notin \{ p_{0},p_{1},\ldots,p_{n} \}$ and $ \pi^{-1}(x) \cong \mathbb{P}^{n-1}$ if $ x \in \{ p_{0},p_{2},\ldots,p_{n}\}.$ Let $ E_{i} \cong \pi^{-1}p_{i}.$ So $  [ E_{i}] \in A_{n-1}(H)= A^{n-1}(H).$ 

Now we can state the result.

\begin{manualtheorem} {C} \label{C}
   (i) For $ k \neq n-1$, $ A^{k}(H)$ is free abelian group with basis 
   $\Big \{ \zeta^{i}\alpha^{j}\medspace |\medspace  i+j = k, 0\leq i\leq n-2,\medspace  0\leq j\leq n \Big\} .$
   
   (ii) $A^{n-1}(H)$ is free abelian group with basis $$\Big \{ \zeta^{i}\alpha^{j} \,\,|\,\,  i+j= n-1, \medspace 0\leq i \leq n-2, \medspace 0\leq j \leq n\Big \} \bigsqcup\Big \{ [E_{i}]\,\, |\,\, 0 \leq i \leq n\Big \}  .$$

 $(iii)$Multiplication in $ A^{\ast}(H)$ is defined by the following relations: 

$(a)$ $ \alpha ^{n+1}= 0$, 

$ (b) \medspace \sum_{j=0}^{n-1} (-1)^{j}\binom{n+1}{j} \zeta^{n-1-j} \cdot \alpha^{j}+ (-1)^{n} \sum_{i=0}^{n}[E_{i}]=0 $

$(c)$ $ [E_{i}]\cdot [E_{j}]= (-1)^{n-1} \delta_{ij} \cdot \zeta^{n-2} \alpha^{n}$

$(d)$ $[E_{i}] \cdot\alpha =0  $ for $ 0\leq i \leq n$

$(e)$ $ E_{i} \cdot\zeta = \alpha^{n} $ for $ 0\leq i \leq n.$
\end{manualtheorem}

Here in $(iii)(c)$,
\[
\delta_{i,j} = 
\begin{cases}
1 & \text{if } i = j \\
0 & \text{if } i \ne j
\end{cases}.
\]

\section{Singularity of cubic hypersurface}
\begin{proof}[Proof of Theorem \ref{A}] If $X$ is not integral, then either $U$ is smooth, or $U$ is an open subset of a normal quadric hypersurface, so $U$ has canonical singularities (since normal quadric hypersurfaces are iterated cones over smooth quadric hypersurfaces, they have canonical singularities by {\cite[Lemma 3.1 (2)] {Ko}}). So we assume $X$  is integral. We proceed by induction on $n$. For $n\leq 2$, there is nothing to show. Suppose $n=3$. If $X$ is normal, Theorem A follows from \cite{BW1}. If $X$ is not normal, by {\cite[Case E]{BW1}}, $X_{\textrm{sing}}$ is a line $L$. Since $U$ is normal, it is regular in codimension $1$. So $U\cap L=\varnothing$. In particular, $U$ is smooth. So Theorem A follows for $n=3$. So we assume $ n\geq 4$.  

Suppose $U$ does not have canonical singularities. Let $x_{0} \in U$ be such that $U$ is not canonical at $x_{0}$. Without loss of generality, we can assume that $x_0=[1:0:\ldots:0].$ Since $X$ is a hypersurface in $\mathbb{P}^n$, $K_{X}$ is Cartier. By {\cite[Corollary 5.24] {KM}}, $U$ does not have rational singularity at $x_{0}$.

\underline{Claim}: For all $y\in X$ with $y\neq x_{0},$ the line joining $y$ to $x_{0}$ is in $X$.

\begin{proof}
    
Let $L$ be the line joining $y$ and $x_0$. Suppose $L\not\subset X.$ Then $ L\cap X$ is a finite set containing $x_{0}$. By shrinking $U$ around $x_{0}$ we may assume $U\cap L= {x_{0}}$ set-theoretically. Let $H$ be a general hyperplane in $\mathbb{P}^n$ containing $L$. So $ H\cap X$ is a cubic hypersurface in $ H \cong \mathbb{P}^{n-1}.$ Since $L \cap U_{\textrm{sm}} = \varnothing,$ the linear system of hyperplanes containg $L$ has no basepoint in $ U_{\textrm{sm}}.$ By Bertini, $ H \cap U_{\textrm{sm}}$ is smooth. Also, $ L\cap U_{\textrm{sing}}= {x_{0}}$ and $H$ is general hyperplane containing $L.$ So, if $F$ is an irreducible component of $ U_{\textrm{sing}}$ such that $ F \neq \{x_{0}\}$, we have $ F\not \subset H.$  So, either $ \textnormal{dim}\medspace (H \cap U_{\textrm{sing}}) =0$ or $ \textnormal{dim}\medspace (H \cap U_{\textrm{sing}}) < \textnormal{dim} \hspace{3pt} U_{\textrm{sing}}.$ Since $ U$ is normal, we have $ \textnormal{dim}\hspace{3pt}U_{\textrm{sing}} \leq \textnormal{dim}\medspace U - 2 = n-3 $. As $n\geq 4$, we get 
$$ \textnormal{dim} \medspace(H \cap U)_{\textrm{sing}}\leq \textnormal{dim}\medspace(H \cap U_{\textrm{sing}}) \leq n-4= \textnormal{dim}\medspace(H\cap U)-2. $$
So, $ H\cap U$ is regular in codimension  1. Since $ H \cap X$ is hypersurface in $H$, so $ H \cap X$ is Cohen-Macaulay, so $S_{2}$. By Serre's criterion of normality, $ H\cap U$  is normal. Since $U$ does not have rational singularity at $x_{0}$, by {\cite[Theorem 5.42]{KM}}, $H\cap U$ does not have rational singularity at $x_0$. By {\cite[Theorem 5.24]{KM}}, $H\cap U$ does not have canonical singularity at $x_0$. By induction hypothesis, $H\cap X $ must be an iterated cone over an elliptic curve. Since $ H \cap X$ is singular at $x_{0}$, $x_{0}$  must be in the linear subspace which is the axis of the cone. So, for all $ \medspace z\in X\cap H$ such that $ z\neq x_{0}$, the line joining $x_{0}$ and $z$ lies in $X \cap H$. Since $y\in H$ by construction, we deduce that $L\in X\cap H$, contradicting our assumption $L\not\subseteq X.$
\end{proof}

By the claim, there is a projective variety $ Y \hookrightarrow \mathbb{P}^{n-1}$ such that $X$ is the cone over $Y$ with vertex $x_{0}.$ Here $ \mathbb{P}^{n-1}$ is the hyperplane in $ \mathbb{P}^n$ consisting of points $[z_{0}:z_{1}:\ldots :z_{n}] $ with $z_{0}=0$.
Since $X$ is normal at $x_{0}$, $Y$ must be projectively normal, a fortiori, normal. Since $X$ is an integral cubic hypersurface, so  $Y$ is also an integral cubic hypersurface. If $ H = \mathcal{O}_{Y}(1),$ adjunction gives $ K_{Y}= -(n-3) H$. Since $ n\geq 4$, we have $ -(n-3) \leq -1.$ Since $X$ is not canonical at $x_{0}$, {\cite[Lemma 3.1 (2)] {Ko}} shows $Y$ is not canonical. By induction hypothesis, $Y$ is an iterated cone over an elliptic curve. Since $X$ is a cone over $Y$, $X$ is also an iterated cone over an elliptic curve. So the proof is complete by induction.    
\end{proof}
\begin{remark}
    The authors came to know later that in the special case when $X$ is normal, Theorem \ref{A} can also be deduced from \cite{beltrametti2015fano}.
\end{remark}
\section{Hyperplane sections of $\mathbb{P}(T_{\mathbb{P}^n})$}
In this section, we shall prove Theorems \ref{B} and \ref{C}. First, we prove a lemma.
\begin{lemma} \label{sing}
     Let $f(X_{0},\ldots X_{n},Y_{0},\ldots Y_{n})$, $g(X_{0},\ldots X_{n},Y_{0},\ldots Y_{n})$ be polynomials, homogenous in $\underline{X}$ and $ \underline{Y}$ separately. Let $$ Z = \Big \{ ( [\underline{x}],[ \underline{y}] \in \mathbb{P}^n \times \mathbb{P}^n \medspace | \medspace f(\underline{x}, \underline{y})= g(\underline{x},\underline{y})  =0 \Big\}.$$
    Assume $Z$ is a variety of dimension $2n-2$. Then $Z_{\textrm{sing}}$ consists of points $([\underline{x}], [\underline{y}])$ such that rk
${\begin{pmatrix}
    \bigtriangledown_{\underline{x}}f & \bigtriangledown_{\underline{y}} f\\
    \bigtriangledown_{\underline{x}}g & \bigtriangledown_{\underline{y}}g
\end{pmatrix}} \leq 1$. 
Here $\bigtriangledown_{\underline{x}}f = \left( \frac{\partial f}{\partial x_{0}}, \cdots ,\frac{\partial f}{\partial x_{n}}\right) $ and $\bigtriangledown_{\underline{y}}f$, $\bigtriangledown_{\underline{x}}g$, $\bigtriangledown_{\underline{y}}g$ are defined similarly. 
\end{lemma}
\begin{proof}
    Define  $ \widetilde{Z}= \Big\{ (\underline{x}, \underline{y})\in (\mathbb{A}^{n+1} \backslash 0)^{2} \medspace | \medspace  f(\underline{x},\underline{y})=g(\underline{x},\underline{y})=0\Big \}$. Then the natural map $ \widetilde{Z} \xlongrightarrow{\pi} Z $ is a principal  $G_{m}^2 $- bundle. So, $ \pi^{-1}  Z_{\textrm{sing}}= \medspace \widetilde{Z}_{\textrm{sing}}.$ By the usual Jacobian criterion of smoothness, for $(\underline{x}, \underline{y}) \in \widetilde{Z}$, we have $(\underline{x}, \underline{y}) \in \widetilde{Z}_{\textrm{sing}} $ if and only if  rk
${\begin{pmatrix}
    \bigtriangledown_{\underline{x}}f & \bigtriangledown_{\underline{y}} f\\
    \bigtriangledown_{\underline{x}}g & \bigtriangledown_{\underline{y}}g
\end{pmatrix}} \leq 1$.  This proves the lemma. One can also prove it by looking at the affine cover of $ \mathbb{P}^n \times  \mathbb{P}^n$ by $ \mathbb{A}^{2n}$ and on each such affine showing that the Jacobian criterion of smoothness of $Z$, is equivalent to the one in the lemma, by applying Euler's theorem on homogeneous polynomials.    
\end{proof}

Now we can prove Theorem \ref{B}.
\begin{proof}[Proof of Theorem \ref{B}]
Let $ X=\mathbb{P}_{\mathbb{P}^n}(T_{\mathbb{P}^n})$. . The line bundle $ \mathcal{O}_{X}(1)$ on $X$ is the restriction of $ \mathcal{O}(1)\boxtimes  \mathcal{O}(1)$ of $ \mathbb{P}^n \times \mathbb{P}^n$. The divisor $X$ on $ \mathbb{P}^n \times \mathbb{P}^n$ corresponds the line bundle $ \mathcal{O}(1)\boxtimes  \mathcal{O}(1)$. So we have an exact sequence 
$$ 0 \longrightarrow \mathcal{O}_{\mathbb{P}^n \times \mathbb{P}^n} \longrightarrow \mathcal{O}(1)\boxtimes  \mathcal{O}(1) \longrightarrow \mathcal{O}_{X}(1) \longrightarrow 0. $$ This gives the following exact sequence on $H^0:$
\begin{equation}\label{scalar}
0 \longrightarrow k \longrightarrow H^{0}( \mathbb{P}^n \times \mathbb{P}^n, \mathcal{O}(1)\boxtimes  \mathcal{O}(1)) \longrightarrow H^{0}(X, \mathcal{O}_{X}(1)) \longrightarrow 0.   
\end{equation}

We identify $$ H^{0}( \mathbb{P}^n \times \mathbb{P}^n, \mathcal{O}(1)\boxtimes  \mathcal{O}(1) )  \cong M_{(n+1)\times (n+1)}$$
\begin{equation*}
  \underline{x}^{t} A \underline{y} \longleftarrow A
  \end{equation*}

 The image of the inclusion $k\rightarrow M_{(n+1)\times (n+1)}$ coming from \eqref{scalar} is the subspace of scalar matrices. So, \eqref{scalar} yields an identification $ H^{0}(X, \mathcal{O}_{X}(1)) = W$, which in turns yields the identification $ | \mathcal{O}_{X}(1) | = \mathbb{P}(W^{\ast}) $ as in the theorem.

 Also, we have
 
 deg $ H_{[\Bar{A}]}= \textrm{deg}\medspace X = \textrm{deg}  \medspace \mathbb{P}^n \times \mathbb{P}^n = \textrm{degree of Segre embedding }= \binom{2n} n$.

Note that by \eqref{H}, if $P\in \textrm{GL}_{n+1}(k)$ and $ J= PAP^{-1}  $  is a Jordan canonical form of $A$,  then the automorphism 
$$ \mathbb{P}^n \times \mathbb{P}^n \longrightarrow \mathbb{P}^n \times \mathbb{P}^n$$
$$ \hspace{35pt} \Big([\underline{x}], [\underline{y}]\Big) \longmapsto \Big([(P^{-1})^t \underline{x}], [P \underline{y}]\Big) $$
gives an isomorphism $ H_{[\Bar{A}]} \xlongrightarrow{\cong}H_{[\Bar{J}]}$.

\textbf{Proof of part 1: }For $ a,b \in \mathbb{Z}$, let  $$ L_{a,b}= \mathcal{O}(a) \boxtimes \mathcal{O}(b) \mid_{X}.$$
One can easily show that $ h^{0}(X, L_{a,b}) =0 $ if $ a <0$ or $ b<0$ (for example, by restricting a section on the fibres of one of the two projections $ X \longrightarrow \mathbb{P}^n).$ So, if the divisor $H_{[\Bar{A}]}$ is the sum of two nonzero effective divisors, they must belong to $|L_{1,0}|$ and $|L_{0,1}|$. Hence,

\hspace{20pt} $ H_{[\Bar{A}]}\hspace{4pt} \textrm{is not integral } \Longleftrightarrow H_{[\Bar{A}]}= D_{1} \cup D_{2} \hspace{4pt} \textrm{for}\hspace{4pt} D_{1}\in |L_{1,0}|, D_{2} \in |L_{0,1}| $ 

\hspace{124pt} $ \Longleftrightarrow H_{[\Bar{A}]}=\Big \{ ([\underline{x}], [\underline{y}]) \in X \hspace{3pt} | \hspace{3pt} \underline{a}^{t}\underline{x}=0 \hspace{3pt} \textrm{or} \hspace{3pt}  \underline{b}^{t} \underline{y}=0 \Big \} \hspace{3pt}\textrm{for some } $

\hspace{195pt} $  \textrm{nonzero}\hspace{3pt} \underline{a},\underline{b} \in k^{n+1}$

\hspace{122pt} $ \Longleftrightarrow H_{[\Bar{A}]}= \Big\{ ([\underline{x}], [\underline{y}]) \in X \hspace{3pt} | \hspace{3pt} \underline{x}^{t}(\underline{a}\underline{b}^{t})     \underline{y}=0 \Big\}\hspace{3pt}\textrm{for some } $

\hspace{185pt} $  \textrm{nonzero}\hspace{3pt} \underline{a},\underline{b} \in k^{n+1} $ 

\hspace{123pt}  $\Longleftrightarrow H_{[\Bar{A}]}= H_{[\underline{a}\underline{b}^{t}]} \hspace{3pt}\textrm{for some  nonzero } \hspace{3pt} \underline{a}, \underline{b} \in k^{n+1}\,\, (\text{by }\eqref{H})$

\begin{equation}\label{rank1}
\hspace{30pt} \Longleftrightarrow [\Bar{A}]= [\Bar{B}] \hspace {3pt} \textrm{ for some rank 1 matrix B}. 
\end{equation}

Suppose \eqref{rank1} holds, we consider 2 cases: 

$ \underline{\textrm{Case 1}}:$ $A$ is diagonalizable.

A Jordan form of $A$ is a scalar multiple of 
\begin{equation}\label{jordan1}
    J_{1} = 
    \begin{bmatrix}
1 & 0 &  \cdots & 0\\
0 & 0 &  \cdots & 0\\
\vdots &\vdots &   \ddots & \\
0 & 0 &  \cdots & 0
\end{bmatrix}_{(n+1)\times (n+1)}, 
\end{equation}
so that $ \underline{x} J_{1} \underline{y}= x_{0} y_{0}.$ Thus we have, 
$$
H_{[\Bar{A}]}\cong H_{[\Bar{J_{1}}]}= D_{1} \cup D_{2},$$ 
where $D_1=\left\{( \underline{[x]} , \underline{[y]}) \in X| x_0=0\right\}, D_2=\left\{( \underline{[x]} , \underline{[y]}) \in X\,\,|\,\, y_0=0\right\}.$ So,
\begin{equation*}
D_{1} \cong D_{2} \cong \mathbb{P}_{\mathbb{P}^{n-1}}\left(T_{\mathbb{P}^n}(-1)\mid_{\mathbb{P}^{n-1}}\right) \cong \mathbb{P}_{\mathbb{P}^{n-1}}( \mathcal{O}_{\mathbb{P}^{n-1}}\oplus T_{\mathbb{P}^{n-1}}(-1)).
\end{equation*}
Here $ \mathbb{P}^{n-1}$ is considered as a hyperplane in $ \mathbb{P}^{n}.$ The intersection 
$$ D_{1} \cap D_{2} = \left\{ \left( \underline{[x]} , \underline{[y]}\right) \in \mathbb{P}^n \times \mathbb{P}^n  \hspace{6pt} | \hspace{3pt} \underline{x}^{t} \underline{y}=0, x_{0}= y_{0}=0\right \}$$

\hspace{110pt} 
$ \cong \Big\{ ( [\underline{z}], [\underline{w}] ) \in \mathbb{P}^{n-1} \times \mathbb{P}^{n-1} \hspace{3pt} | \hspace{3pt} \underline{z}^{t} \underline{w} = 0 \Big\} $

\vspace{3pt}
\hspace{110pt} $ \cong \mathbb{P}_{\mathbb{P}^{n-1}}(T_{\mathbb{P}^{n-1}})$

\vspace{3pt}
\hspace{110pt} $ = \hspace{3pt}V_{0,n}$. 

Also, since $ H_{[\Bar{A}]} \cong H_{[\Bar{J_{1}}]}$ , it follows that all such $ H_{[\Bar{A}]}$'s are isomorphic.

$ \underline{\textrm{Case 2}}:$ $A$ is not diagonalizable. 

A Jordan form of $A$ is a scalar multiple of 
\begin{equation}\label{jordan2}
    J_{2} = 
\begin{bmatrix}
0 & 1 &  \cdots & 0\\
0 & 0 &  \cdots & 0\\
\vdots &\vdots &   \ddots & \\
0 & 0 &  \cdots & 0
\end{bmatrix}_{(n+1)\times (n+1)}, 
\end{equation}

so that $ \underline{x}^{t} J_{2} \underline{y}= x_{0}y_{1}.$ Hence we have

$$ H_{[\Bar{A}]}= D_{1} \cup D_{2}$$ 
where $ D_{1} \cong D_{2} \cong \mathbb{P}_{\mathbb{P}^{n-1}}( \mathcal{O}_{\mathbb{P}^{n-1}} \oplus T_{\mathbb{P}^{n-1}}(-1))$ as in case $1$. We also have,
$$ D_{1} \cap D_{2} = \left\{( \underline{[x]} , \underline{[y]}) \in \mathbb{P}^n \times \mathbb{P}^n  \hspace{3pt} | \hspace{3pt} \underline{x}^{t} \underline{y}=0, x_{0}= y_{1}=0 \right\} $$
\vspace{3pt}
\hspace{125pt} $ \cong \Big\{ ( [\underline{z}],[\underline{w}]) \in \mathbb{P}^{n-1} \times \mathbb{P}^{n-1} \hspace{3pt}| \hspace{3pt} \sum_{1<i \leq n} z_{i}w_{i} =0 \Big \} $

\vspace{3pt}
\hspace{117pt} $ = \hspace{3pt}V_{1,n}.$

Also, since $ H_{[\Bar{A}]} \cong H_{[\Bar{J_{2}}]}$ , it follows that all such $ H_{[\Bar{A}]}$'s are isomorphic. 

In both cases, $D_{1}$ and $D_{2}$ are conjugate modulo some automorphism of $ \mathbb{P}^n \times \mathbb{P}^n$, so $D_{1}$, $ D_{2}$ have same degree, which is 
$$ \frac{1}{2} \textrm{deg} H_{[\Bar{A}]}= \frac{1}{2} \binom {2n} n .$$     
It is also clear that $ H_{[\Bar{A}]}$ is always reduced. This proves the first part of the theorem.

\textbf{Proof of part 2:} Let $\Bar{A}\in W\setminus{0}$ be such that $H_{[\Bar{A}]} $ is irreducible. The map 
$$ \mathcal{O}_{\mathbb{P}^n}^{n+1} \xlongrightarrow{s} \mathcal{O}_{\mathbb{P}^n}(1)^{2}$$
\hspace{190pt} $ \underline{y} \longmapsto \begin{pmatrix}
   \underline{X}^{t} \underline{y} \\
    \underline{X}^{t}A \underline{y}
\end{pmatrix}$

is generically surjective, as $ A$ is not scalar matrix. So, $ E = \textrm{Ker} (s)$ is generically a vector bundle of rank $ n-1$. Letting $ H= H_{[\Bar{A}]}$, consider the map 
$H \longrightarrow \mathbb{P}^n $ given  by 
$  \left(\underline{[x]} , \underline{[y]}\right) \longmapsto [\underline{x}].$  
By \eqref{H}, over an open set $U$ in $ \mathbb{P}^n$, both $H$ and $ \mathbb{P}(E^{\ast})$ are projective subbundles of the trivial projective bundle $ U \times \mathbb{P}^n$ defined by the equations $$ \underline{x}^{t} \underline{y}=\underline{x}^{t}A \underline{y}=0  .$$ 

So, $H$ and $ \mathbb{P}(E^{\ast})$ are birational, hence $ H$ is rational.  

 Now let us prove the statement about $H_{\textrm{sing}}$. We can assume without loss of generality that $A$ is in Jordan form. By Lemma \ref{sing}, for $ ([\underline{x}], [ \underline{y}]) \in H$, we have

\hspace{17pt} $ ([\underline{x}], [ \underline{y}]) \medspace \in  H_{\textrm{sing}} \Longleftrightarrow \textrm{rk} {\begin{pmatrix}
    \underline{y}^{t} & \underline{x}^{t} \\
    \underline{y}^{t} A^{t} & \underline{x}^{t}A
\end{pmatrix}}\leq 1   $

\hspace{102pt} $ \Longleftrightarrow  \text{there is }\hspace{2pt} \lambda \in k \hspace{3pt}\textrm{such that} \hspace{3pt} \underline{y}^{t} A^{t}= \lambda \underline{y}^{t},\,\, \underline{x}^{t} A = \lambda \underline{x}^{t} $

\hspace{102pt} $ \Longleftrightarrow \text{there is } \hspace{2pt} \lambda \in k \hspace{3pt}\textrm{such that} \hspace{3pt} (A-\lambda I)\underline{y} =0,\,\, \underline{x}^{t}(A-\lambda I)=0 $

 \hspace{102 pt} $ \Longleftrightarrow \Big([\underline{x}],[ \underline{y}]\Big) \in  \textrm{ Sing}_{\lambda} A \hspace{3pt} \textrm{for some eigenvalue } \lambda \hspace{3pt} \textrm{of} \hspace{3pt} A. $

 Here for a nonscalar matrix $B \in M_{(n+1)\times (n+1)}(k)$, and $ \lambda \in k$, $$ \textrm{Sing}_{\lambda} (B) := \Big\{ ( [\underline{x}], [ \underline{y}]) \in \mathbb{P}^n \times \mathbb{P}^n \hspace{2pt} | \hspace{2pt} \underline{x}^{t} \underline{y} = \underline{x}^{t}B \underline{y}=0, \underline{x}^{t}(B-\lambda I)=0, (B-\lambda I) \underline{y}=0\Big\}. $$ 
 So, if $\lambda_{i}$'s are the distinct eigenvalues of $A$, we have 
 $$H_{\textrm{sing}} = \sqcup_{i} \hspace{2pt}\textrm{Sing}_{\lambda_{i}}(A) .$$ 
 The following claim completes the proof of $$H_{\textrm{sing}} = \sqcup_{i} V_{s_{i},r_{i}}.$$

 $ \textbf{ \underline{Claim}}$: $ \textrm{\textrm{Sing}}_{\lambda_{i}} \hspace{2pt} A\cong V_{s_{i}, r_{i}} $

  \begin{proof} Noting that $ \textrm{Sing}_{\lambda_{i}} \hspace{2pt}(A)= \textrm{Sing}_{0} \hspace{3pt} ( A - \lambda_{i} I)$, without loss of generality, we can assume that $ \lambda_{i}=0$. Let 
 \begin{equation*}
    A = 
\begin{bmatrix}
A_{0} & 0 \\
0     & A_{1} 
\end{bmatrix}, 
\end{equation*}
where $A_{0}$  is the union of the Jordan blocks with eigenvalues 0, and $ A_{1}$ is the union of the Jordan blocks with eigenvalues nonzero. Let $ B_{1}, B_{2}, \ldots ,B_{r} $ be the Jordan blocks of $A$ with eigenvalue $0$. Assume $ B_{1}, B_{2}, \ldots ,B_{s} $ have size greater than $1$, and $ B_{s+1},\ldots ,B_{r}$ have size exactly $1$. Let  $ B_{i}$ have size $ l_{i}$, so $ B_{i} \in M_{l_{i}\times l_{i}}.$ Also, given $ \underline{x}, \underline{y} \in k^{n+1}$, write   $$ \underline{x} = ( \underline{z},\underline{x_{1}}),    \hspace{2pt} \textrm{where} \hspace{2pt} \underline{z}= (\underline{z_{1}}, \ldots \underline{z_{r}} ), \hspace{3pt} \underline{z_{i}} \in k^{l_{i}} \text{ for all } i $$
$$\underline{y}= ( \underline{w}, \underline{y_{1}}), \quad \textrm{where}\,\,\, \underline{w}= ( \underline{w_{1}}, \ldots, \underline{w_{r}}), \hspace{5pt}\underline{w_{i}} \in k^{l_{i}} \text{ for all } i .$$

We  have 

\hspace{57pt}$  \left([\underline{x}], [\underline{y}]\right) \in \textrm{Sing}_{0}(A) \Longleftrightarrow \underline{x}^{t}\underline{y}=0,\,\,  \underline{x}^{t}A=0,\,\, A\underline{y}=0$

 $$ \Longleftrightarrow \underline{x}^{t}\underline{y}=0, \,\,\,\underline{x_1}^{t} A_1=0,\,\,\, A_1 \underline{y_1}=0,\,\,\, B_i \underline{w_{i}}=0,\,\,\, \underline{z_{i}}^t B_i=0\quad \text{for}\,\,\, 1\leq i\leq s, \medspace \underline{x_{1}}=0, \,\,\underline{y_{1}}=0  $$
$$\hspace{1.8pt} \Longleftrightarrow  \underline{x_{1}}= \underline{y_{1}}=0,\,\,\, z_{i,1}=\ldots=z_{i,l_{i}-1}=0,\,\,\, w_{i,2}=\ldots= w_{i,l_{i}}=0 \quad\text{ for } 1\leq i\leq s, \sum_{s<i\leq r}z_{i}w_{i}=0. $$

So, $$ \textrm{Sing}_{0}\hspace{2pt}(A) \cong \Big \{([z_{1l_{1}}:z_{2l_{2}}:\ldots :z_{sl_{s}}:z_{s+1}:\ldots : z_{r}],[w_{11}:w_{21}:\ldots :w_{s1}:w_{s+1}:\ldots: w_{r}])$$
\begin{equation*}
\hspace{-90pt}\in \mathbb{P}^{r-1} \times \mathbb{P}^{r-1} \hspace{2.5pt} | \hspace{2.5pt}   \sum_{s<i\leq r} z_{i}w_{i}=0 \Big\}
\end{equation*}
\hspace{50pt} $ \cong V_{s,r}.$
\end{proof}

If $r_i\geq n$ for some $i$, then for some $\lambda\in k$, a Jordan form of $A$ has at least $n$ Jordan blocks corresponding to eigenvalue $\lambda$. So, $A-\lambda I$ is either zero or has \eqref{jordan1} or \eqref{jordan2} as a Jordan form. But this is impossible by part $1$ as $H_{[\Bar{A}]} $ is irreducible. So, for each $i$, we have $r_i\leq n-1$, hence dim $V_{s_i,r_i}\leq 2r_i-2\leq 2n-4$ . So, we have $ \textrm{dim }H_{\textrm{sing}} \leq 2n-4. $ So $H$ is regular in codimension $1$. $H$ is a hypersurface in the smooth variety $X$, so it is $S_{2}$. By Serre's criterion of normality, $H$ is normal. By adjunction, $ K_{H} = L_{-(n-1), -(n-1)}|_{ H}$ is anti-ample, so $H$ is Fano. 

For $ i,j \in \{ 0,1,\ldots ,n\}$, let  $ U_{ij} \cong \mathbb{A}^{2n}$ be the open affine in $ \mathbb{P}^n \times \mathbb{P}^n$ defined by $ x_{i} \neq 0, y_{j} \neq 0$. Note that  $$ X \subseteq \bigcup_{i\neq j} U_{ij}$$  
as the points outside the union are the points  $p_{i}$,  where $p_{i}$ has $ x_{j}=y_{j}=0 \hspace{3pt} \text{for all}  \hspace{3pt} j \neq i$, none of the $ p_{i}$'s lie in $X$. So, $ U_{ij}$'s gives an open cover of $ H_{[\Bar{A}]}$ in $ \mathbb{P}^n \times \mathbb{P}^n.$ It suffice to show that $ H_{[\Bar{A}]} \cap U_{ij}$ is a hyperplane in $ \mathbb{A}^{2n-1}$ of degree less than equal to 3. Without loss of generality, we assume $ i=0, j=1. $ On $ U_{01} \cong \mathbb{A}^{2n},$ the equation $ \underline{x}^{t} \underline{y}=0 $ gives $ y_{0} + x_{1}+ x_{2}y_{2}+ \ldots + x_{n}y_{n}=0$ (put $x_{0}=1, y_{1}=1)$. So, 
$$ y_{0}= q(x_{1},x_{2},\ldots, x_{n},y_{2}, \ldots, y_{n}) ,$$ 
where $q$ is a polynomial of degree 2. Substituting this in the equation $ \underline{x}^{t} A\underline{y}=0$ and putting $ x_{0}=y_{1}=1$, we get  an equation $ P(x_{1},x_{2},\ldots ,x_{n},y_{2}, \ldots y_{n})=0 $, with $ \textrm{deg}\hspace{2pt}P \leq 3 .$
So 
$$ H \cap U_{ij} \cong  \Big\{ ( x_{1},\ldots, x_{n}, y_{2},\ldots , y_{n}) \in \mathbb{A}^{2n-1} \hspace{3pt} | \hspace{3pt} P(x_{1},\ldots x_{n}, y_{2},\ldots,y_{n})=0 \Big\}, $$ 
hence $ H\cap U_{ij}$ is isomorphic to a hypersurface in $ \mathbb{A}^{2n-1}$ of degree less than or equal to 3. We have already shown that $H$ is normal rational, so the hypersurface is too. The projective closure of this hypersurface is rational, so it cannot be iterated cone over an elliptic curve. Hence, by Theorem \ref{A}, the hypersurface has canonical singularities, so $H$ has canonical singularities.

\textbf{Proof of part 3:} For $ n\geq 3$ and a general $ [\Bar{A}]$ in $ \mathbb{P}(W^{\ast})$, $H_{[\Bar{A}]}$ is smooth by Bertini, and 
$$ \textrm{Pic}(X) \longrightarrow \textrm{Pic}(H_{[\Bar{A}]})$$ 
 is an isomorphism by the Grothendieck-Lefschetz theorem. So, $ H= H_{[\Bar{A}]}$ is a rational smooth Fano variety of Picard rank 2, and  by $(3)$, the maps $$ p,q : H \longrightarrow \mathbb{P}^n$$  
 has all fibers projective spaces, where $$ p\Big([\underline{x}],[\underline{y}]\Big ) = [\underline{x}] \quad\textrm{and}\quad  q\Big([\underline{x}],[\underline{y}]\Big ) = [\underline{y}].$$ 
 So, $ p.q$ are the two elementary contractions of $H.$ Clearly, dimension of a general fiber of $p$ or $q$ is $ \textrm{dim} \hspace{3pt} H-n = n-2$, hence the general fiber is $ \mathbb{P}^{n-2}$. This proves part $3$.

\end{proof}

\begin{remark}
Our method gives an alternate proof of the main theorem in \textrm{\cite{MN}}. All the parts in that theorem, except the description of the singularity as $A_1$ or $A_2$ type, are part of our theorem. For $n=2$, one can reduce $A$ in part 2 of our theorem to Jordan form and explicitly compute the equations of the hypersurfaces as stated in part 2 to show that the singularity is of $A_1$ or $A_2$ type. 
\end{remark}

\begin{remark}\label{deform}
    Since $\mathbb{P}(T_{\mathbb{P}^n})$ is the rational homogeneous space $\textrm{SL}_{n+1}(\mathbb{C})/B$ where $B$ is a parabolic subgroup, as in \textrm{\cite[Remark 2.8] {MN}}, Theorem B gives 
    \begin{enumerate}
\item All possible degeneration of a certain Fano manifold of Picard rank $2$ in the rational homogeneous space $\textrm{SL}_{n+1}(\mathbb{C})/B$,
\item All possible deformations of the union of two Schubert divisors in $\textrm{SL}_{n+1}(\mathbb{C})/B$.
\end{enumerate}
\end{remark}
\begin{corollary}\label{dual}
Regard $\mathbb{P}(T_{\mathbb{P}^n})$ as a subvariety of $\mathbb{P}(W)$ via the embedding given by $| \mathcal{O}_{\mathbb{P}(T_{\mathbb{P}^n})}(1) |$. Then the dual variety of $\mathbb{P}(T_{\mathbb{P}^n})\subseteq \mathbb{P}(W) $ is 
$$\Big\{[\Bar{A}]\in \mathbb{P}(W^{\ast})\,\,|\,\, A \text{ does not have distinct eigenvalues}\Big\}.$$
\end{corollary}
\begin{proof}
    If $H_{[\Bar{A}]}$ is smooth, it must be irreducible. From Theorem \ref{B}, it is clear that $(H_{[\Bar{A}]})_{\textrm{sing}}=\varnothing$ if and only if $V_{s_i, r_i}=\varnothing$ for each eigenvalue $\lambda_i$ of $A$. This is equivalent to saying that $s_i=0,$ $r_i=1$ for each eigenvalue $\lambda_i$ of $A$. This is the same as saying $A$ has distinct eigenvalues.
\end{proof}
\vspace{4pt}

\begin{proof}[Proof of Theorem \ref{C}] The parts
    $(i)$  and $(ii)$ are immediate  from  \textrm{\cite[Theorem 3.1, eq. 3.3]{Jia}}. We identify $ A^{2n-2}(H) $ with $ \mathbb{Z}$, $ \zeta^{n-2}\alpha^{n}$ being the canonical generator.
 We prove $(iii)$. $ (a)$ is clear, $(d)$ follows as $ E_{i}\cap \pi^{-1}L = \varnothing $   for a general hyperplane in $L$ in $ \mathbb{P}^n$.

$ \underline{Proof\medspace of \medspace(c)}$: For $ i \neq j$ we have  $E_{i} \cap E_{j} = \varnothing $, hence $[E_{i}] \cdot [E_{j}]=0$.  For smooth varieties $ Y_{1} \hookrightarrow Y_{2}$, let $ N_{Y_{1}/Y_{2}}$ denote normal bundle of $ Y_{1}$ in $ Y_{2}$. Letting $ X= \mathbb{P}(T_{\mathbb{P}^n})$, we have following short exact sequence 
$$ 0 \longrightarrow N_{E_{i}/H}\longrightarrow N_{E_{i}/X}\longrightarrow N_{H/X|_{E_{i}}} \longrightarrow 0. $$ Note that $N_{H/X|_{E_{i}}}= \mathcal{O}_{X}(H)_{|_{E_{i}} }= \mathcal{O}_{\mathbb{P}(T_{\mathbb{P}^n})}(1)_{|_{E_{i}}}= \mathcal{O}_{E_{i}}(1), N_{E_{i}/X}= \mathcal{O}_{E_{i}}^{n}.$ So, $ N_{E_{i}/H}= \Omega_{E_{i}}(1)$ by Euler exact sequence in $ E_{i}\cong \mathbb{P}^{n-1}.$ We have 
$$ [E_{i}]^{2}= c_{n-1}(N_{E_{i}/H})= c_{n-1}(\Omega_{\mathbb{P}^{n-1}}(1))= (-1)^{n-1}.$$ 

    $\underline{Proof \medspace of\medspace (b)}:$  By $(ii)$, there are unique integers $  a_{0},\ldots, a_{n-2}, \beta_{0},\ldots, \beta_{n}$  such that 
    \begin{equation} \label{chow}
 \zeta^{n-1}= \sum_{j=1}^{n-1} a_{n-1-j} \zeta^{n-1-j}\alpha^j + \sum_{j=0}^{n} \beta_{j}[E_{j}].
 \end{equation}
Note that $ \zeta^{n-1} \cdot [E_{j}]= ( \zeta_{| E_{j}})^{n-1}= c_{1}(\mathcal{O}_{E_{j}}(1))^{n-1}=1.$ So,  (\ref{chow}) and $(d)$, $(c)$ gives 
$$ 1= \zeta^{n-1} \cdot[E_{j}]= \beta_{j}[E_{j}]^{2}=(-1)^{n-1}\beta_{j}.$$ So, $ \beta_{j}= (-1)^{n-1}$ for each $j$.
Let $ \Tilde{\zeta}=c_{1}(\mathcal{O}_{\mathbb{P}(T_{\mathbb{P}^n})}(1)), \Tilde{\alpha}= q^{\ast}c_{1}(\mathcal{O}_{\mathbb{P}^n}(1))\in A^{1}(X).$ Note that for $ 0\leq i \leq 2n-2,$ we have  
$$ \pi_{\ast}(\zeta^{i})= q_{\ast} i_{\ast}(\zeta^{i})= q_{\ast}(\Tilde{\zeta}^{i+1}) \,\,\,(\text{by \cite[Proposition 2.6 (b)]{Fu}})$$ $$\,= (-1)^{i-n} s_{i+2-n}(T_{\mathbb{P}^n}),$$ where $s_{j}(T_{\mathbb{P}^n})$ is the $j^{th}$ Segre class of $ T_{\mathbb{P}^n}$ (see {\cite[Section 3.1]{Fu}}, the convention of projective bundle there is slightly different from the convention in {\cite[Chapter 2, section 7]{Hart}} which we are following, so the sign $ (-1)^{i-n}$ is coming). Let $ \alpha_{1}= c_{1}(\mathcal{O}_{\mathbb{P}^n}(1)) \in A^{1}(\mathbb{P}^n).$

Using \textrm{\cite[Theorem 3.1]{Jia}} and \textrm{\cite[Equation 3.4]{Jia}}, we have for $ 0 \leq i \leq n-2 $,
$$ a_{i}\alpha_{1}^{n-1-i}= \pi_{i \ast}(\zeta^{n-1})= \sum_{j=0}^{n-2-i}(-1)^{j}c_{j}(T_{\mathbb{P}^n}) \cap \pi_{\ast}(\zeta^{n-2-i-j}\cdot \zeta^{n-1}) $$
$$ \hspace{75pt}= \sum_{j=0}^{n-2-i}(-1)^j c_{j}(T_{\mathbb{P}^n})\cdot \pi_{\ast}(\zeta^{2n-3-i-j})$$
$$\hspace{140pt} =\sum_{j=0}^{n-2-i}(-1)^{j}c_{j}(T_{\mathbb{P}^n}) \cdot s_{n-1-i-j}(T_{\mathbb{P}^n})(-1)^{n-1-i-j} $$
$$ \hspace{100pt}= (-1)^{n-1-i} \sum_{j=0}^{n-2-i}c_{j}(T_{\mathbb{P}^n})\cdot s_{n-1-i-j}(T_{\mathbb{P}^n}) $$
$$ \hspace{10pt}= (-1)^{n-i}c_{n-1-i}(T_{\mathbb{P}^n}), $$ as  $\sum_{j=0}^{n-1-i}c_{j}(T_{\mathbb{P}^n}) \cdot s_{n-1-i-j}(T_{\mathbb{P}^n}) =0.$

By the Euler exact sequence 
$$ 0 \longrightarrow \mathcal{O}_{\mathbb{P}^n} \longrightarrow \mathcal{O}_{\mathbb{P}^n}(1)^{n+1} \longrightarrow T_{\mathbb{P}^n} \longrightarrow 0 ,$$
we get $$ c_{n-1-i}(T_{\mathbb{P}^n})= \binom{n+1}{i+2} \alpha_{1}^{n-1-i}.$$
So, $ a_{i}= (-1)^{n-i} \binom{n+1}{i+2}$ for $ 0\leq i \leq n-2.$ This proves $(b).$

$ \underline{Proof \medspace of \medspace (e)}$: First we show that $\zeta \cdot [E_{k}]$ is independent of $k$. By (i), it suffices to show that for each $0\leq i\leq n-2$,
\begin{equation*}
\zeta^{i}\cdot \alpha^{n-2-i}\cdot[E_{k}] 
\end{equation*}
is independent of $k$. This is indeed true, since  $\zeta^{i}\cdot\alpha^{n-2-i}\cdot\zeta\cdot[E_{k}] = \zeta_{|E_{k}}^{i+1}(\alpha_{|E_{k}})^{n-2-i}$
\[=
\left\{
\begin{array}{ll}
1 & \text{if}\,\, i = n-2,\\ 
0 & \text{otherwise, } 
\end{array}
\right.
\] as $\zeta_{|E_{k}} = c_{1}(\mathcal{O}_{E_{k}}(1)), \alpha_{|_{E_{k}}}  = 0$.
So, there is $\tau \in A^{n}(H)$ such that $\zeta\cdot[E_{k}] = \tau$ for each $0\leq k\leq n$. 

Since $c_{j}(T_{\mathbb{P}^{n}}) = \binom{n+1}{j}\alpha_{1}^{j}$ for $0\leq j \leq n$, we have 
\begin{equation*}
\tilde{\zeta}^{n} = \sum^{n}_{j = 1} (-1)^{j-1} \binom{n+1}{j} \tilde{\zeta}^{n-j} \tilde{\alpha}^{j},    
\end{equation*}
by {\cite[Appendix A]{La}}. Restricting this to $H$, we get 
\begin{equation}\label{Top intersection}
{\zeta}^{n} = \sum^{n}_{j = 1} (-1)^{j-1} \binom{n+1}{j} {\zeta}^{n-j} {\alpha}^{j}
\end{equation}
Now (\ref{chow}) and (\ref{Top intersection}), gives
\begin{equation*}
\sum^{n}_{j = 1} (-1)^{j-1} \binom{n+1}{j} \zeta^{n-j} \alpha^{j} = \sum_{j = 1}^{n-1} (-1)^{j-1} \binom{n+1}{j} \zeta^{n-j}\alpha^{j} + (-1)^{n-1} \sum_{j = 0}^{n} [E_{j}] \cdot \zeta,   
\end{equation*}
i.e., $(-1)^{n-1}(n+1)\alpha^{n} = (-1)^{n-1}(n+1)\tau$, that is, $\tau = \alpha^{n}$, proving (e).
\end{proof}

\printbibliography
\end{document}